\documentclass{amsart}
\usepackage[english]{babel}
\usepackage{amssymb,amsmath,xcolor,mathrsfs}
\usepackage{hyperref,color}
\usepackage{graphicx}
\usepackage{wrapfig}
\usepackage[utf8]{inputenc}
\usepackage{float}
\usepackage{wrapfig}
\usepackage[all]{xy}
\usepackage{fancyhdr}
\numberwithin{equation}{section}
 
\newtheorem{theorem}{Theorem}[section]
\newtheorem{definition}{Definition}

\newtheorem{lemma}{Lemma}

\newtheorem{corollary}{Corollary}[section]

\newcommand{\Sec}{Sect\,}

\newcommand{\di}{{\rm{ div_\Sigma}}}
\newcommand{\g}{{\rm{\nabla \hspace{0.4mm}}}}

\newcommand{\Ric}{{\rm{Ric}}}

\newcommand{\Hess}{{\rm{Hess\,}}}
\newcommand{\area}{\rm{|\Sigma|}}

\title[On the first Stability Eigenvalue of Surfaces ]
{ On the first Stability Eigenvalue of Surfaces with 
Constant weighted mean curvature
}
\author{M. Batista, J. I. Santos}
\address{IM, Universidade Fe\-deral de Alagoas, Macei\'o, 
AL, CEP 57072-970, Brazil}
\email{mhbs@mat.ufal.br \mbox{and} jissivan@gmail.com}

\thanks{The first author was supported by CNPq/Brazil. }

\begin{document} 
\maketitle
\begin{abstract} 
Let $\Sigma$ be a compact immersed surface with constant weighted mean curvature 
$H_f$ in a weighted manifold $(M^3,g,f)$. In this paper we obtain upper bounds for the
first eigenvalue of the weighted Jacobi operator  on $\Sigma$ in terms of $H_f$ and the 
curvature of the ambient. As consequence we obtain that there is no stable self-shrinker 
of the mean curvature flow.

\end{abstract}

\section{Introduction}
Many branches of Differential Geometry, such as, Ricci flow, mean curvature flow
and optimal transportation theory, are related to 
Riemannian manifolds endowed of a smooth positive density function
(see for instance \cite{e, cmz, mw1, mw2, ww, m, gp, v} and references therein). 
A theory of these spaces and its  
curvatures goes back to Lichnerowicz \cite{l1, l2} and more recently  by Bakry and \'Emery 
\cite{BE}, and it  has been very active area in recent years.

We recall that a {weighted manifold} is a 
Riemannian manifold $({M^3},g)$ 
endowed with a real-valued smooth function $f: M \to \mathbb{R}$ 
which is used as a density to measure geometric objects on $M$ .
Associated to this structure
we have an important second order differential operator defined by
$$\Delta_f u= \Delta u - \langle\nabla u, \nabla f\rangle,$$
where $u \in C^\infty$. This operator is know as the Drift Laplacian.

Also, following  \cite{wylie} and \cite{BE}, the natural generalizations of
the sectional and Ricci curvatures  are defined as 
\begin{equation}\label{102}\overline{\Sec}_f^{2m}(X,Y)= \overline{\Sec}(X,Y) + \frac{1}{2}
\left(\Hess f(X,X)-\frac{(df(X))^2}{2m}\right),\end{equation}
and
\begin{equation}\label{104}\Ric_f^{2m}=\Ric_f -\frac{df\otimes df}{2m},\end{equation}
where $X$ and $Y$ are unit and orthogonal vectors fields tangents to $M$, $m>0$,
and $\Ric_f=\Ric + \Hess f.$

\medskip



Now, we  introduce some objects 
related with the theory of isometric immersions in a weighted manifold.
Let  $\psi:\Sigma^2\rightarrow  M^3$ be an
isometric immersion of an oriented surface $\Sigma$ in $M^3$
and consider $N$ an unit normal vector field
globally defined. We will denote by $A$ its {second fundamental form} 
and by $H$ the {mean curvature} of $\Sigma$, that is, the trace of $A$. 


We recall that the {weighted mean
curvature}, introduced by Gromov in \cite{g}, of the immersion $\psi$ is given by 
$$H_{f}=H+\langle N,\nabla f\rangle.$$

Along this paper, $dv_f = e^{-f}dv$ will denote  
the weighted area element  of the surface $\Sigma$, 
where $dv$ is the area element of $\Sigma$,
and $\area$ denote the weighted area of $\Sigma$. Furthermore, we will denote by $K$ the
Gaussian curvature of $\Sigma$ and $\overline\Sec_\Sigma$ is the sectional
curvature of $M$ restricted to $\Sigma.$  

It is a remarkable fact that, in the variational context,
immersed surfaces with constant  weighted mean curvature are critical points of the 
weighted area functional under variations that preserves the
weighted volume (see \cite{b}). 
Moreover, the second variation of the weighted area gives rise the
weighted Jacobi operator on  $\Sigma$, see \cite{e}, which is defined by 
\begin{equation}\label{equation 161}
J_{ f}u=\Delta_{ f} u+(| A|^2+\Ric_{ f}(N,N))u, 
\end{equation}
for any $u\in C^{\infty}(\Sigma)$ and $|A|^2$ is the Hilbert-Schmidt norm of $A$.

\medskip

%
%

In this paper, encouraged by the ideas in \cite{ P, ABP, ABB},  we  investigate some
geometric aspects of immersed surfaces with constant weighted mean curvature.
In particular  we study problems related with the first  eigenvalue of the
weighted Jacobi operator on compact surfaces without boundary.

\medskip

Our first result reads as follows:

\begin{theorem}\label{186}
Let $(M^3,g, f)$ be a weighted manifold with $\Ric_{ f}^{2m}\geq2c$ 
and  $\overline{\Sec}\geq c$, for some $c\in\mathbb R$, and consider $\psi:\Sigma^2\rightarrow M^3$
an isometric immersion of a compact surface with constant weighted mean curvature
$H_{ f}$. Let $\lambda_1$ be the first eigenvalue of weighted Jacobi operator. Then,  
\begin{enumerate}
\item [\bf{(i)}] $\lambda_1\leq-\dfrac{1}{2}\left(\frac{H_{ f}^2}{1+m}+4c\right)$, 
with equality if and only if $\Sigma$ is totally
umbilical in $M^3$, $\Ric_{ f}^{2m}=2c$ and $d f(N)=\dfrac{m}{1+m}H_{ f}$ on $\Sigma$;
\item [\bf{(ii)}] $\lambda_1\leq-\dfrac{H_{ 
f}^2}{(1+2m)}-4c-\dfrac{8\pi(g-1)}{\area}$, where $g$ denotes the genus of $\Sigma$. 
Furthermore,
 the equality occurs if and only if  $K$ is constant, $\overline{\Sec}_{\Sigma}=c$,
$\Ric_{ f}^{2m}=2c$ and $d f(N)=\dfrac{2m}{1+2m}H_ f$ on $\Sigma$.
\end{enumerate}
\end{theorem}
\medskip

As a consequence of this theorem we obtain

\begin{corollary} Under the  conditions of Theorem \ref{186} we set 
$$
\Lambda_1=-\left(\dfrac{H_{f}^2}{1+2m}+4c\right)
\, \textrm{ and }  \,
\Lambda_2=-\dfrac{1}{2}\left(\frac{H_{ f}^2}{1+m}+4c\right).$$
Assume $H_f^2 \geq -4(1+m)(1+2m)c$. If $\Lambda_1 < \lambda_1\leq \Lambda_2$, then $g=0$,
that is, $\Sigma$ is topologically a sphere.
\end{corollary}


\medskip

The second result is the following:

\begin{theorem}
Let $(M^3,g, f)$ be a weighted manifold with 
$\overline{\Sec}_f^{2m}\geq c$, for some $c\in \mathbb R$ 
and $Hess f\leq  \sigma\cdot g$ for some real 
function $\sigma$ on $M$.Consider $\psi:\Sigma^2\rightarrow M^3$ an 
isometric immersion of a compact surface with constant weighted mean curvature
$H_{ f}$. Then,
\begin{enumerate}
\item [\bf{(i)}] $\lambda_1\leq-\frac{1}{2}\left(\frac{H_{ f}^2}{1+m}+4c\right)$, 
with equality if and only if $\Sigma$ is totally
umbilical in $M^3$, $\Ric_{ f}^{2m}=2c$ and $d f(N)=\dfrac{m}{1+m}H_{ f}$ on $\Sigma$;

\item [\bf{(ii)}] $\lambda_1\leq-\dfrac{H_{f}^2}{(1+2m)}-\left(4c-\dfrac{\int_\Sigma \sigma\, dv_f}{\area}\right)-\dfrac{8\pi(g-1)}{\area},$  where $g$ denotes the genus of $\Sigma$. 

If  the equality occurs, then $\overline{\Sec}_f^{2m}=c$, 
$\Ric_{f}^{2m}=2c$, $d f(N)=\dfrac{2m}{1+2m}H_ f$, 
$H=\frac{1}{1+2m}H_f$, and $| A|$ is a constant on $\Sigma$. 
Moreover, 
$M^3$ has constant sectional curvature $k$ and $e^{-f}$ 
is the restriction of a coordinate function from the  appropriate
canonical embedding of $\mathbb{Q}^3_k$ in $\mathbb{E}^4$, 
where $\mathbb{E}^4$ is $\mathbb{R}^4$ or $\mathbb{L}^4.$
\end{enumerate}

 \end{theorem}
 
 \medskip

Now, we will introduced the notion of stability  aiming to announce some consequences.  
 
 \begin{definition}
 Under the above notation. We say that a surface $\Sigma$ is stable 
 if the first eigenvalue $\lambda_1$ of the Jacobi operator is nonnegative. 
 Otherwise, we say that $\Sigma$ is unstable.
  \end{definition}
 
 Our second consequence is:

\begin{corollary}
Under the assumptions of the {Theorem \ref{186}}. 
 \begin{enumerate}
  \item[{(i)}] If $c>0$, then $\Sigma$ cannot be stable;
  \item[{(ii)}] If $c=0$, but $H_f\neq0$,  then $\Sigma$ cannot be stable;
  \item[{(iii)}] If $c=H_f=0$, and the genus $g\geq2$, then $\Sigma$ cannot be stable;
  \item[{(iv)}] If $c=0$ and $\Sigma$ is stable, then $H_f=0$.
 \end{enumerate}
\end{corollary}

\medskip

For the next result we recall that a Self-shrinker of the mean curvature flow is an oriented surface 
$x: \Sigma \to \mathbb{R}^3$ such that 
$$H=-\frac{1}{2}\langle x,N\rangle,$$
where $N$ is an unit normal vector field on $\Sigma$. So, if we consider $\mathbb{R}^3$ endowed 
with the function $f(x)=\frac{|x|^2}{4},$ then a Self-shrinker is a $f$-minimal surface in the Euclidean space.

\medskip

The next result is a consequence of the proof of the Theorem \ref{186} and it reads as follows:

\begin{corollary}
All compact  self-shrinker in the 3-dimensional Euclidean space is unstable.
\end{corollary}

More generally, the triple $(\mathbb{R}^3,\delta_{ij}, |x|^2/4)$ is known as {\it Gaussian space} 
and using this notation we have that

\begin{corollary}
All compact surface in the Gaussian space with constant weighted mean curvature  is unstable.
\end{corollary}


\section{Preliminaries}

\medskip

An important result for us is the classification of  weighted manifolds with constant weighted sectional curvature.
The result below follows closely the one in \cite{wylie}, and we include the proof here for 
the sake of completeness.

\begin{lemma}\label{261}
Let $(M^3,g,f)$ be a weighted manifold. Assume that $\overline{\Sec}_f=c$, then $M$
has constant sectional curvature $k$, for some $k\in\mathbb R$. Moreover, $c=-(m-1)k$ and if $f$ is a non constant function, then $u=e^{-f/m}$ is the restriction of a coordinate function from the appropriate canonical embedding of a space form of curvature $k$, $\mathbb{Q}^3_k$, in $\mathbb{E}^4$, where $\mathbb{E}^4$ is $\mathbb{R}^4$ or $\mathbb{L}^4.$
\end{lemma}
\begin{proof}
Let $X$ and $Y$ be an unit and orthogonal vectors on $M$. Then, by equation (\ref{102}), we get
$$c= \overline{\Sec}(X,Y) + \frac{1}{2}
\left(\Hess f(X,X)-\frac{(df(X))^2}{2m}\right)$$
and
$$c= \overline{\Sec}(Y,X) + \frac{1}{2}
\left(\Hess f(Y,Y)-\frac{(df(Y))^2}{2m}\right).$$

So, there exists a smooth function $w: M\to \mathbb{R}$ such that 
$$\Hess f-\frac{df\otimes df}{2m} = w\cdot g.$$
Then, letting $\{E_1, E_2, X\}$ be an orthonormal frame we have
$$2c = \sum_{i=1}^2\overline{\Sec}_f(X,E_i)=\Ric(X,X) +2w.$$

Thus, by Schur's Lemma, $w$ is a constant function and so $M$ has constant 
sectional curvature, say $k$. Defining the function 
$u=e^{-f/m}$, we have that
\begin{equation}\label{278}\Hess u = -\dfrac{c-k}{m}u\cdot g.\end{equation}
So, by Lemma 1.2 in \cite{t},
\begin{equation}\label{285} g = dt^2 + (u')^2g_0,\end{equation}
where $g_0$ is a local metric on a surface orthogonal to $\nabla u$ (a level set of $u$) and $u'$ denotes
the derived of $u$ in the direction of the gradient  of $u$.

Computing the  radial sectional curvature of the metric (\ref{285}), we have  $(c+(m-1)k)u'=0.$ 
Since $f$ is non constant, we have that $c=-(m-1)k.$ Moreover, as $u$ satisfies
equations (\ref{278}) and (\ref{285}), $u$ is the restriction of a coordinate function from the 
appropriate  canonical embedding of $\mathbb{Q}^3_k$ in $\mathbb{E}^4$, 
where $\mathbb{E}^4$ is $\mathbb{R}^4$ or $\mathbb{L}^4.$

\end{proof}

Now we consider a first eigenfunction $\rho\in C^{\infty}(\Sigma)$ of the Jacobi
operator $J_f$, that is,
$J_{ f}\rho=-\lambda_1\rho$;
or  equivalently, 
\begin{equation}\label{182}
 \Delta_{ f}\rho= (\lambda_1+| A|^2+\Ric_{ f}(N,N))\rho.
\end{equation}
Furthermore, $\lambda_1$ is simple and it is characterized  by  
\begin{equation}\label{192} 
 \lambda_1=\inf\left\{\frac{-\int_{\Sigma}uJ_{f}u\, dv_f}{\int_{\Sigma}u^2\, dv_f}:\,u\in 
C^{\infty}(\Sigma),\,u\neq0\right\}.
\end{equation}

We observe that the first eigenfunction of an elliptic second-order  differential
operator has a sign. Therefore, without loss of generality,  we can assume that
$\rho >0.$ In what follows, we emphasize that every integrations are with respect
to the weighted area element.

\medskip

Thus, 
\begin{eqnarray*} \Delta_{ f}\ln\rho&=&\Delta\ln\rho-\langle\g f,\g\!\ln\rho\rangle\\
&=&\di_{}(\g\!\ln\rho)-\langle\g f,\rho^{-1}\g\rho\rangle\\
&=&\di(\rho^{-1}\g \rho)-\rho^{-1}\langle\g f,\g\rho\rangle\\
&=&\rho^{-1}\di(\g\rho)+\langle\g\rho^{-1},\g \rho\rangle-\rho^{-1}\langle\g f,\g\rho\rangle\\
&=&\rho^{-1}(\Delta\rho-\langle\g f,\g\rho\rangle)-\rho^{-2} |\g \rho|^2\\
&=&\rho^{-1}\Delta_{ f}\rho-\rho^{-2}|\g\rho|^2\\
&=&-(\lambda_1+| A|^2+\Ric_{ f}(N,N))-\rho^{-2}|\g\rho|^2.
\end{eqnarray*}
Integrating the equality above on $\Sigma$ and using the divergence theorem we have that
$$
\alpha:=\int_{\Sigma}\rho^{-2}|\g\rho|^2= - \int_{\Sigma}(| A|^2+\Ric_{ f}(N,N)) - \lambda_1\area,
$$
where $\alpha\geq 0$ defines 
a simple invariant that is independent of the choice
of $\rho$, because $\lambda_1$ is simple. In other words
$$
\lambda_1=-\frac{1}{\area}\left(\alpha+\int_{\Sigma}(| A|^2+\Ric_{ f}(N,N))\right).
$$ 


Let $\{E_i\}$  be an orthonormal frame in $T\Sigma$ and $\{a_{ij}\}$ the 
coefficients of $A$ in the frame, 
using the Gauss equation
$$K=\overline{\Sec}_{\Sigma}-\langle A(X),Y\rangle^2+
\langle A(X),X\rangle \langle A(Y),Y\rangle,$$
we have that 
$$
K-\overline{\Sec}_{\Sigma} = a_{11}a_{22} - a_{12}^2 = 
\dfrac{1}{2}\left((a_{11}+a_{22})^2 - \sum_{i,j=1}^2a_{ij}^2\right) 
= \dfrac{1}{2} \left(H^2 - |A|^2\right),
$$
hence
\begin{equation}\label{232}
| A|^2=H^2+2(\overline{\Sec}_{\Sigma}-K).
\end{equation} 
Using the Gauss\,-Bonnet theorem and the definition of weighted
mean curvature, $H_{ f}=H+\langle N,\g f\rangle$, we obtain that
$$
 \lambda_1=-\frac{1}{\area}\left\{\alpha-8\pi(1-g)+
 \int_{\Sigma}[H^2+2\overline{\Sec}_{\Sigma}+\Ric_{ f}(N,N)]\right\},
 $$
 that is,
$$
\lambda_1=-\frac{1}{\area}\left\{\alpha-8\pi(1-g)+\int_{\Sigma}(H_{ f}-\langle
N,\g f\rangle)^2+
\int_{\Sigma}[2\overline{\Sec}_{\Sigma}+\Ric_{ f}(N,N)]\right\}.
$$
Moreover, we know that for all  $a,b\in\mathbb R$ and $m>-1$, it holds  that
\begin{equation}\label{249}
(a+b)^2\geq\frac{a^2}{1+m}-\frac{b^2}{m}, 
\end{equation}
with equality if and only if  $b=-\frac{m}{1+m}a.$

Applying this inequality in the term $(H_{ f}-\langle N,\g f\rangle)^2$  and using the definition of
$
\Ric_f^{2m},
$
we obtain that 
$$ 
\hspace{-0.25cm}\lambda_1\leq -\frac{1}{\area}\left\{\alpha+8\pi(g-1)+
\int_{\Sigma}\left[\frac{H_{ f}^2}{1+2m}-\frac{\langle
N,\g f\rangle^2}{2m}\right]+ \int_{\Sigma}\left[2\overline{\Sec}_{\Sigma}+
\Ric_{ f}(N,N)\right]\right\},
$$
i.e., 
\begin{equation}\label{308}
\hspace{1cm}\lambda_1\leq -\frac{H_{ f}^2}{1+2m}-\frac{1}{\area}\left\{\alpha+8\pi(g-1)+
\int_{\Sigma}\left(\Ric_{ f}^{2m}(N, N)+2\overline{\Sec}_{\Sigma}\right)\right\}.
\end{equation}

To finalize this section, we recall the traceless of the second fundamental form of $\Sigma$, 
that is, the  tensor $\phi$ defined by $\phi=A-\frac{H}{2}I$,
where $I$ denotes the  identity endomorphism on $T\Sigma$. 
We note that $tr(\phi)=0$ and
$|\phi|^2=| A|^2-\frac{H^2}{2}\geq0$, 
with equality if and only if $\Sigma$ is totally umbilical, where
$ | \phi |^2$ is the Hilbert-Schmidt norm.

In the literature, $\phi$  is know as the total umbilicity 
tensor of $\Sigma$. In terms of  $\phi$, the Jacobi operator is rewritten as  
\begin{equation}\label{139}
J_{ f}u=\Delta_{ f} u+\left(|\phi|^2+
\frac{H^2}{2}+\Ric_{ f}(N,N)\right)u.
\end{equation}

We use exactly this expression in next section to obtain a estimate 
of the first eigenvalue of the weighted Jacobi operator.

\section{Main Results}
\begin{theorem}\label{274}
Let $(M^3,g, f)$ be a weighted manifold with $\Ric_{ f}^{2m}\geq2c$ and 
 $\overline{\Sec}\geq c$. Consider $\psi:\Sigma^2\rightarrow M^3$ an 
isometric immersion of a compact surface  with constant weighted mean curvature
$H_{ f}$. 
Then, 
\begin{enumerate}
\item [\bf{(i)}] $\lambda_1\leq-\frac{1}{2}\left(\frac{H_{ f}^2}{1+m}+4c\right)$, with 
equality if and only if $\Sigma$ is totally
umbilical in $M^3$, $\Ric_{ f}^{2m}=2c$ and $d f(N)=\dfrac{m}{1+m}H_{ f}$ on $\Sigma$.
\item [\bf{(ii)}] $\lambda_1\leq-\dfrac{H_{ 
f}^2}{(1+2m)}-4c-\dfrac{8\pi(g-1)}{\area}$, where $g$ denotes the genus of $\Sigma$. 
Furthermore,
the equality occurs if and only if  $K$ is a constant, $\overline{\Sec}_{\Sigma}=c$,
$\Ric_{ f}^{2m}=2c$ and $d f(N)=\dfrac{2m}{1+2m}H_ f$ on $\Sigma$.
\end{enumerate}
\end{theorem}
\begin{proof}{(i)}
 Choosing the constant function $u=1$  to be the test function  in 
$(\ref{192})$ to estimate $\lambda_1$, and using the expression in $(\ref{139})$, we obtain that
\begin{align*}
 \lambda_1&\leq\frac{-\int_{\Sigma}1J_f1}{\int_{\Sigma}1^2}=-\frac{1}{
\area}\left[\int_{\Sigma}|\phi|^2+\frac{1}
{2}\int_{\Sigma}H^2+\int_{\Sigma } \Ric_ { f } (N , N)\right] \\
&=-\frac{1}{\area}\left[\int_{\Sigma}|\phi|^2+\frac{1}{2}\int_{\Sigma}(H_{ f}-\langle
N,\g f\rangle)^2+\int_{\Sigma}\Ric_{ f } (N , N)\right]\\
&\leq-\frac{1}{\area}\left[\int_{\Sigma}|\phi|^2+\frac{1}{2}\int_{
\Sigma}\left(\frac{H_{ f}^2}{1+m}-\frac{\langle 
N,\g f\rangle^2}{m}\right)+\int_{\Sigma}\Ric_{ f } (N , N)\right]\\
&=-\frac{H_{f}^2}{2(1+m)}-\frac{1}{\area}\left[\int_{\Sigma}
|\phi|^2+\int_{ \Sigma}\Ric_{ f}^{2m}(N,
N)\right]\\
&\leq-\frac{H_{f}^2}{2(1+m)}-2c-\frac{1}{\area}\int_{\Sigma}
|\phi|^2\\
&\leq -\frac{1}{2}\left(\frac{H_{ f}^2}{1+m}+4c\right).
\end{align*}

If $\lambda_1=-\frac{1}{2}\left(\frac{H_{ f}^2}{1+m}+4c\right)$, then all the inequalities  above becomes equalities and consequently
$\phi=0$ on $\Sigma$, $\Ric_{ f}^{2m}=2c$ and $d f(N)=\dfrac{m}{1+m}H_{ f}$. Therefore, $\Sigma$ is totally umbilical and the
normal of $\Sigma$ is a direction of minimal to curvature $\Ric_{ f}^{2m}$.
\vspace{3mm}

On the other hand, if $\Sigma$ is totally umbilical,  $\Ric_{ f}^{2m}=2c$ and
$d f(N)=\dfrac{m}{1+m}H_{ f}$, we have
\begin{eqnarray*}
 H&=&H_{ f}-d f(N)\\
&=&H_{ f}-\dfrac{m}{1+m}H_{ f}\\
&=&\dfrac{1}{1+m}H_{ f}
\end{eqnarray*}
and
$$ 2c=\Ric_{ f}^{2m}(N,N)=\Ric_{ f}(N,N)-\frac{1}{2m}(d f(N))^2.$$
So
\begin{eqnarray*}
\Ric_{ f}(N,N)&=&2c+\frac{1}{2m}(d f(N))^2\\
&=&2c+\frac{m}{2(1+m)^2}H_{ f}^2.
\end{eqnarray*}
Hence,
\begin{eqnarray*}
J_{ f}&=&\Delta_{ f}+\frac{H^2}{2}+2c+\frac{m}{2(1+m)^2}H_{ f}^2\\
&=&\Delta_{ f}+\frac{1}{2(1+m)^2}H_{ f}^2+2c+\frac{m}{2(1+m)^2}H_{ f}^2\\
&=&\Delta_{ f}+\frac{1}{2(1+m)}H_{ f}^2+2c,
\end{eqnarray*}
and thus we conclude that
$$\lambda_1=-\frac{1}{2}\left(\frac{H_{ f}}{1+m}+4c\right).$$
\vspace{5mm}

\noindent{ (ii)} As $\Ric_{ f}^{2m}\geq 2c,\,\overline{\Sec}\geq c$ and $\alpha\geq0$, we obtain by equation
 $(\ref{308})$ that
$$\lambda_{1}\leq-\frac{H_{ f}^2}{1+2m}-4c-8\pi(g-1)/\area.$$

If the equality occurs, then $\alpha=0,\,\overline{\Sec}_{\Sigma}=c$, $\Ric_{ f}^{2m}=2c$, 
this last means that the normal direction of $\Sigma$ is a
direction of minimal of the curvature $\Ric_{ f}^{m}$ of $M^3$. 
Moreover, we obtain of  the equation $(\ref{249})$ that
$$d f(N)=\frac{2m}{1+2m}H_{ f}.$$

In fact, $\alpha=0$ implies that $\rho$ is a constant and using of the equation $(\ref{182})$ we have that $|
A|^2$ is also a constant. Futhermore,
$$
 d f(N)=\frac{2m}{1+2m}H_{ f}
$$
implies that the mean curvature
$$
H=H_{ f}-\frac{2m}{1+2m}H_{ f}=\frac{1}{1+2m}H_{ f}
$$
is a constant and by equation $(\ref{232})$ we have that $K$ is a constant.\vspace{3mm} 

On the other hand, if $K$ is a constant, $\overline{\Sec}_{\Sigma}=c,\,\Ric_{ f}^{2m}=2c$ and $d f(N)=\dfrac{2m}{1+2m}H_{ f}$, we have that
$$
\Ric_{ f}(N,N)=2c+\frac{2m}{(1+2m)^2}H_{ f}^2,\qquad H=\frac{1}{1+2m}H_{ f},
$$
and so
\begin{eqnarray*}
J_{ f}&=&\Delta_{ f}+| A|^2+\Ric_{ f}(N.N)\\
&=&\Delta_{ f}+ H^2+2(c-K)+2c+\frac{2m}{(1+2m)^2}H_{ f}^2\\
&=&\Delta_{ f}+\frac{H_{ f}^2}{(1+2m)^2}+4c-2K+\frac{2m}{(1+2m)^2}H_{ f}^2\\
&=&\Delta_{ f}+\frac{H_{ f}^2}{1+2m}+4c-2K,
\end{eqnarray*}
and this implies that
$$
\lambda_{1}=-\frac{H_{ f}^2}{1+2m}-4c+2K.
$$
Now, using the Gauss-Bonnet theorem we conclude that
$$
\lambda_{1}=-\frac{H_{ f}^2}{1+2m}-4c-\frac{8\pi(g-1)}{\area}.
$$
\end{proof}\vspace{3mm}

Now considering that the ambient is a $3$-dimensional simply connected space form with sectional curvature $c$, $\mathbb{Q}^3_c$. If $c$ is positive, we assume that all surfaces are contained in a hemisphere. 
In this conditions we obtain the follows  result: 

\begin{corollary}
Let $\psi:\Sigma^2\rightarrow \mathbb{Q}_c^3$ be an isometric immersion of a compact and orientable surface 
with constant weighted mean curvature $H_{ f}$ and let $f$ be one half of the square of the extrinsic distance
function. Assume that $\psi(\Sigma)$ is contained in the geodesic ball center in the origin $0$ and radius $\sqrt{m}$. 
Then
\begin{enumerate}
\item[(i)] $\lambda_1\leq-\dfrac{1}{2}\left(\dfrac{H_{ f}^2}{1+m}+4c\right)$ 
\medskip
\item[(ii)]  $\lambda_1\leq-\dfrac{H_{ f}^2}{(1+2m)}-4c-\dfrac{8\pi(g-1)}{\area}. $
\end{enumerate}
The equalities occur if and only if $\psi(\Sigma)$ is the sphere center in the origin and radius $\sqrt{m}$.
\end{corollary}
\begin{proof}
 We know that
$$ \Ric_{ f}^{2m}(N,N)=2c+ \Hess r^2(N,N)-\frac{(dr^2(N))^2}{2m}$$
and 
\begin{eqnarray*}
 \Hess r^2(N,N)&=&\langle\nabla_N\g r^2,N\rangle\\
&=&2(dr(N))^2+2r\Hess r(N,N).
\end{eqnarray*}
Now, using the expression of the Hessian of the distance function in a space form, we have that
$$\Hess r(N,N)= \cot_c(r)[1-(dr(N))^2],$$
where
$$
cot_c(s)=\left\{ \begin{array}{ccl}
\sqrt{-c}\, coth(\sqrt{-c}s) & \mbox{if} & c<0, \\
\frac{1}{s} & \mbox{if} & c=0,\\
\sqrt{c}\, cot(\sqrt{c}s) & \mbox{if} & c>0.
\end{array}\right.
$$
So,
$$\Hess r^2(N,N)=2(dr(N)^2)+2r\cot_c(r)[1-(dr(N))^2].$$

\medskip
Now, using that the surface is contained in the ball center in the
origin and radius $\sqrt{m}$ and $(dr(N))^2\leq1$, we obtain that
\begin{eqnarray*}
 \Ric_{ f}^{2m}(N,N)&=&2c+2(dr(N))^2+2r\,\cot_c(r)(1-(dr(N))^2)-\frac{4r^2(dr(N))^2}{2m}\\ 
&\geq&2c.
\end{eqnarray*}
Therefore by  {Theorem \ref{274}}, we conclude the inequalities enunciates. 
To finalize, if the equalities occurs in the 
 inequalities, then $dr(N)=1$ and $r^2=m.$

\end{proof}

\begin{corollary}\label{equation 5}
 Let $(M^3,g, f)$ be a weighted manifold with $\Ric_{ f}^{2m}\!\geq 2c$ and $\overline{\Sec}\geq c$.  
\begin{enumerate}
 \item[{(i)}] There is no compact stable surface with 
$$\frac{H_{ f}^2}{1+2m}+4c>0.$$ 
 \item[{(ii)}] If $\Sigma^2$ is a compact and stable surface and
$\dfrac{H_{ f}^2}{1+2m}+4c=0$, then $\Sigma^2$
is topologically  a sphere or a torus.
 \item[{(iii)}] If $\Sigma^2$ is a compact and stable surface such that 
$\dfrac{H_{ f}^2}{1+2m}+4c<0$, then
$${\area}\left|\dfrac{H_{ f}^2}{1+2m}+4c\right|\geq 8\pi(g-1).$$ 
\end{enumerate}
\end{corollary}
\begin{proof}
 By definition, a surface is stable if and only if $\lambda_1\geq0 $. Thus the item 
{\ (i)} follows from the {Theorem \ref{274}} (i).  For the item {(ii)}, we used that $\frac{H_{ f}^2}{1+2m}+4c=0$  and Theorem \ref{274} (ii) to obtain that
$$\lambda_1\leq -\frac{8\pi(g-1)}{\area}.$$
Now, using the stability we conclude that $g=0$ or $1$ and thus $\Sigma$ is topologically a sphere or a torus.

To finalize we using the definition of stability and the Theorem \ref{274} (ii). So,
$$0\leq\lambda_1\leq -\dfrac{H_{ f}^2}{1+2m}-4c-\frac{8\pi(g-1)}{\area},$$
and thus
$${\area}\left|\dfrac{H_{ f}^2}{1+m}+4c\right|\geq 8\pi(g-1).$$
\end{proof}


Another consequence of the {Theorem \ref{274}} is improves the proposition $3.2$ in \cite{ir} for the case in that $\Sigma$ is not necessarily $f$-minimal.
\begin{corollary}\label{equation 6}
 Under the same assumptions of the {Theorem \ref{274}}. 
 \begin{enumerate}
  \item[{(i)}] If $c>0$, then $\Sigma$ cannot be stable;
  \item[{(ii)}] If $c=0$, but $H_f\neq0$,  then $\Sigma$ cannot be stable;
  \item[{(iii)}] If $c=H_f=0$, and the genus $g\geq2$, then $\Sigma$ cannot be stable;
  \item[{(iv)}] If $c=0$ and $\Sigma$ is stable, then $H_f=0$.
 \end{enumerate}
\end{corollary}
Note that the item {(i)} and {(ii)} of the { Corollary \ref{equation 6}}  are only a rewrite of the item 
{ (i)} of {Corollary \ref{equation 5}} and the item {(iii)} and {(iv)} are rewrite of the item {(ii)}.\vspace{3mm}

Now, we recall the generalized  sectional
curvature 
$$
\overline{\Sec}_f^{2m}(X,Y)=\overline{\Sec}(X,Y)+\frac{1}{2}
\left(\Hess f(X,X)-\frac{(df(X))^2}{{2m}}\right),
$$
where $X$, $Y$ are unit and orthogonal vectors fields on $M$.

\medskip
Moreover,
$$
\Ric_f^{ 2m}(X,X)=\sum_{i=1}^{2}\overline{\Sec}_f^{2m}(X,Y_i).
$$

Since 
$$\Ric_f^{2m}(N,N)+2\overline{\Sec}_{\Sigma}\geq \Ric_f^{2m}(N,N)+2\overline{\Sec}_f|_\Sigma-\Hess f(X,X),$$
where $X$ is a vector field on $\Sigma$.
So, if $\Hess f\leq \sigma g$, we rewrite the expression $(\ref{308})$ by

\begin{eqnarray} \label{569}\\
\lambda_1&\leq&-\frac{H_{f}^2}{1+2m}-\frac{1}{\area}\left\{
\alpha+8\pi(g-1)+\int_{\Sigma}\left(\Ric_{f}^{2m}(N, N)+2\overline\Sec_f^{2m}-\sigma\right)\right\}.\nonumber
\end{eqnarray}\\
\medskip
As a consequence, we have the following result:

\begin{theorem}\label{equation 4}
 Let $(M^3,g, f)$ be a weighted manifold with the weighted sectional  curvature satisfying
  $\overline{\Sec}_f^{2m}\geq c$, for some $c\in\mathbb R$, and $\Hess f\leq  \sigma\cdot g$ for some real function $\sigma$ on $M$.
Consider $\psi:\Sigma^2\rightarrow M^3$ an isometric immersion of a compact surface 
with constant weighted mean curvature $H_{ f}$. 
Then, 
\begin{enumerate}
\item [{(i)}] $\lambda_1\leq-\frac{1}{2}\left(\frac{H_{ f}^2}{1+m}+4c\right)$, 
with equality if and only if $\Sigma$ is totally
umbilical in $M^3$, $\Ric_{ f}^{2m}=2c$ and $d f(N)=\dfrac{m}{1+m}H_{ f}$ on $\Sigma$.

\item [{(ii)}] $\lambda_1\leq-\dfrac{H_{f}^2}{(1+2m)}-\left(4c-\dfrac{\int_\Sigma \sigma\, dv_f}{\area}\right)-\dfrac{8\pi(g-1)}{\area},$  
where $g$ denotes the genus of $\Sigma$. 

If  the equality occurs, then $\overline{\Sec}_f^{2m}=c$, 
$\Ric_{f}^{2m}=2c$, $d f(N)=\dfrac{2m}{1+2m}H_ f$, $H=\frac{1}{1+2m}H_f$, and $| A|$ is a constant on $\Sigma$. 
Moreover, 
$M^3$ has constant sectional curvature $k$ and $e^{-f}$ 
is the restriction of a coordinate function from the appropriate 
canonical embedding of $\mathbb{Q}^3_k$ in $\mathbb{E}^4$, 
where $\mathbb{E}^4$ is $\mathbb{R}^4$ or $\mathbb{L}^4.$
\end{enumerate}
\end{theorem}

\begin{proof} The item (i) is a consequence of Theorem \ref{274} (i). To second item, we using the expression in (\ref{569}) and our hypotheses.

Now, if equality holds, then $\alpha=0$, $\Ric_{ f}^{2m}=2c$ and $\overline{\Sec}_f^{2m}=c$. 
 By equality in the inequality $(\ref{249})$, we obtain 
$$d f(N)=\frac{2m}{1+2m}H_{ f},$$
and so
$$ H=H_{ f}-\frac{2m}{1+2m}H_{ f}=\frac{1}{1+2m}H_{ f}.$$
Moreover, $\alpha=0$ imply that $\rho$ is constant and of  the equation $(\ref{182})$ we have that $|
A|^2$ is also a constant.  

To finish, we using the Lemma \ref{261}  to conclude that $M^3$ has constant sectional curvature
and $e^{-f}$ has the property enunciate in case of the equality.
\end{proof}
\section*{Acknowledgement}

We would like to thank the M. Cavalcante for useful suggestions and comments.

%


\end{document}